\documentclass[10pt,doublespace]{article}
\usepackage{psfrag}
\usepackage{amssymb,bbm}
\usepackage{amsmath}
\usepackage{amsfonts}
\usepackage{pstricks,pstricks-add}
\usepackage{graphicx}
\usepackage[latin5]{inputenc}
\usepackage{yfonts}
\usepackage{setspace}
\usepackage[toc,page]{appendix}
\newcounter{theorem}
\newtheorem{theorem}{Theorem}

\newtheorem{lemma}{Lemma}

\newtheorem{example}{Example}

\newtheorem{remark}{Remark}
\newtheorem{definition}{Definition}

\newenvironment{proof}[1][Proof]{\textbf{#1.} }{\rule{0.5em}{0.5em}}

\title{TUBE FORMULA FOR SELF-SIMILAR FRACTALS WITH NON-STEINER-LIKE GENERATORS \date{}}

\author{Ali DENİZ \thanks{Corresponding Author.}  \footnote{Anadolu
University, Science Faculty, Department of Mathematics, 26470,
Eskişehir, Turkey,
\hspace{5cm}e-mails:$\quad$
adeniz@anadolu.edu.tr, skocak@anadolu.edu.tr, yunuso@anadolu.edu.tr,
aeureyen@anadolu.edu.tr}
 \and Şahin KOÇAK $^*$ \and Yunus ÖZDEMİR $^*$ \and A. Ersin
ÜREYEN $^*$}

\begin{document}
\maketitle 

\pagestyle{myheadings} \markboth{Tube Formula for Self-Similar Fractals}{Tube Formula for Self-Similar Fractals}

\begin{abstract}
We give a direct, pointwise proof for the tube formula of Lapidus-Pearse for self-similar fractals, where we allow non-convex, non-Steiner-like generators.
\end{abstract}

\textbf{Keywords:}{ Self-similar fractals, tube formula, complex dimensions,
 self-similar tiling, zeta functions.}
\section{Introduction}

M. Lapidus and E. Pearse proved in \cite{lape1} a tube formula for
higher dimensional fractals, extending the earlier work of Lapidus
and his coworkers \cite{LaFra}. They associate certain tilings
with fractals and express the volume of the inner
$\varepsilon$-neighborhood of the tiling in terms of residues of
a certain associated function, called the geometric
$\zeta$-function.  The residues thereby are taken at the complex
dimensions of the fractal, a notion introduced and elaborated by
Lapidus and his coworkers.

A recent work  of E. Pearse and S. Winter \cite{pewin} clarifies
the relationship between the inner $\varepsilon$-neighborhood of
the tiling and the genuine $\varepsilon$-neighborhood of the
fractal in a very satisfactory way: If the boundary of the convex hull of the fractal is a subset of the fractal, then the volume of the $\varepsilon$-neighborhood of the fractal is
the sum of the volumes of the inner $\varepsilon$-neighborhood of the
tiling and the outer $\varepsilon$-neighborhood of the convex hull
of the fractal.

As the volume of the outer $\varepsilon$-neighborhood of the
convex hull is rather trivial, the tube formula in terms of
residues gives effectively the true volume of the
$\varepsilon$-neighborhood of the fractal if the Pearse-Winter
condition is fulfilled. This circumstance attributes a higher
value to the utility of the tube formula.

The proof of the tube formula of Lapidus-Pearse is distributional
and they assume Steiner-like generators. Moreover, as their theory
goes beyond the self-similar fractals, the proof of the general
tube formula for fractal strings is rather involved.

In \cite{bizim}, we generalized part of their theory to the
graph-directed fractals, defined complex dimensions for them and
gave a general scheme of summation yielding tube formulas for
self-similar as well as graph-directed fractals.

In the present note we want to give a direct and pointwise proof
for the tube formula of L-P for self-similar fractals, whereby we
consider non-convex and non-Steiner-like generators also.

In the next section we very briefly recall the approach  of L-P
and in the third section we give the proof of the tube formula.

\section{The Tube Formula for Self-Similar Tilings}

Let \[F=\bigcup \limits_{j=1}^J \Phi_j(F)=: \Phi (F) \subset
\mathbb{R}^d \] be a self-similar fractal, where
$\Phi_j:\mathbb{R}^d \rightarrow \mathbb{R}^d$ are similitudes
with scaling ratios $0<r_j<1$, $j=1,\dots,J$. Consider the convex
hull $C:=[F]$ of the fractal (for which we assume $\dim C=d$). We
assume that the system $\{\Phi_j\}$ satisfies the open set
condition with ${\rm int \, } C $ a feasible open set. This condition
is also called the tileset condition in \cite{pewin}. Additionally we
assume the non-triviality condition of \cite{pewin}:
\[{\rm int\,}C \nsubseteq \Phi(C)= \bigcup \limits_{j=1}^J
\Phi_j(C).\] This condition amounts to ${\rm int\,} F=\emptyset$.

Now define $T_1={\rm int}(\,C \setminus \Phi (C) ) $ and its iterates $T_n=\Phi^{n-1}
(T_1)$, $n=2,3,\dots $ (see \cite{pearse}).

The tiling of the self-similar system is given by \[\mathcal{T}:=\left
\{ T_n \right \}_{n=1}^{\infty}\] and the volume of the inner $\varepsilon$-neighborhood
 of the tiling $\mathcal{T}$ is defined by
\[V_{\mathcal{T}}(\varepsilon):=\sum_{n=1}^{\infty}
V_{T_n}(\varepsilon),
\] where $V_{T_n}(\varepsilon)$ is the volume of the inner
$\varepsilon$-neighborhood of $T_n$.

 To state the tube formula we need some additional assumptions and definitions.
 Assume that $T_1$ is the union of finitely many (connected) components,
 $T_1=G_1 \cup G_2 \cup \cdots \cup G_Q$, called the generators of the tiling.
 In the work of L-P they assume the generators to be  Steiner-like in the following
 sense:

 A bounded, open set $G \subset \mathbb{R}^d$ is called (diphase) Steiner-like
if the volume $V_G(\varepsilon)$ of the inner
$\varepsilon$-neighborhood of $G$ admits an expression of the form
\[ V_{G}(\varepsilon)= \sum_{i=0}^{d-1}
\kappa_i(G)\varepsilon^{d-i},\qquad \text{ for } \varepsilon <g,
\] where  $g$ denotes the  inradius of $G$, i.e. supremum of the radii
of the balls contained in $G$.

For $\varepsilon \geq g$ we have
$V_{G}(\varepsilon)={\rm{volume}}(G)$ which is denoted by
$-\kappa_d(G)$, the negative sign being conventional \cite{lape1}.

To be precise, L-P consider also the ``pluriphase'' case, where $V_G( \varepsilon)$
is given piecewise by different polynomials in the region $0< \varepsilon < g$. But
 this generalization brings no essential complication.

Lapidus-Pearse introduce the following ``scaling
$\zeta$-function'':
\begin{definition}\label{scalingzeta}
The scaling $\zeta$-function of the self-similar fractal is
defined by
\[\zeta(s)=\sum_{k=0}^\infty \sum_{w \in W_k}r_w^s,\] where $W_k$
is the set of words $w=w_1w_2\cdots w_k$ of length $k$ (with
letters from $\{1,2,\dots ,J\}$) and $r_w=r_{w_1}r_{w_2}\dots
r_{w_k}$.
\end{definition}

A simple calculation shows that $\zeta(s)$ can be
expressed as \cite[Theorem 2.4]{LaFra}
\begin{equation}
\label{moran} \zeta(s)=\frac{1}{1-\sum_{j=1}^Jr_j^s} \quad \text{
for } {\rm {Re}} (s)> D,
\end{equation}
where $D$ is the \emph{similarity dimension} of the system (i.e. the
unique real root of the Moran equation $1-\sum_{j=1}^Jr_j^s=0$ which
coincides with the Minkowski and Hausdorff dimensions if the open-set
condition holds). We assume $d-1 < D <d$.

\begin{definition}
The set $\mathfrak{D}:=\{ \omega \in \mathbb{C} \mid \zeta (s)  \text {
has a pole at } \omega \}$ is called the set of complex dimensions
of the self-similar fractal.
\end{definition}

Lapidus-Pearse define a second type of ``$\zeta$-functions'' associated
with the tiling and related to the geometry of the (diphase) Steiner-like
generators. As the case of multiple generators does not bring additional complication, we assume that there is a single generator $G$.

\begin{definition}
 The geometric $\zeta$-function $\zeta_{\mathcal{T}}(s,\varepsilon)$ associated
 with the generator $G$ is defined by \[\zeta_{\mathcal{T}}(s,\varepsilon):=
\zeta(s) \varepsilon ^{d-s}  \sum_{i=0}^d
\frac{g^{s-i}}{s-i}\kappa_i(G).\]
\end{definition}

The formula of Lapidus-Pearse for
$V_{\mathcal{T}}(\varepsilon)$ now reads as follows:

\begin{theorem}[Distributional tube formula for tilings of self-similar fractals, \cite{lape1}]
\begin{equation}
V_{\mathcal{T}}
(\varepsilon)=\sum_{\omega \in \mathfrak{D}_{\mathcal{T}}}
{\rm{res}} (\zeta_{\mathcal{T}}(s,\varepsilon);\omega),
\end{equation}
where $\mathfrak{D}_{\mathcal{T}}=\mathfrak{D} \cup \{0,1,\dots,d-1\}$.
\end{theorem}

In this work however, we want to do a first step into the
non-Steiner-like realm. We shall assume the volume
$V_G(\varepsilon)$ of a generator $G$ to have the following form:
\begin{equation} \label{vegeepsilon}
 V_G(\varepsilon)=\left \{ \begin{array}{ccc}
\displaystyle \sum_{i=0}^{d-1} \kappa_i(G)\varepsilon^{d-i} &\text{ for } 0<\varepsilon < h   \\
\displaystyle  \lambda_G(\varepsilon) & \text{ for } h\leq \varepsilon \leq g , \\
\end{array} \right .
\end{equation} where $\lambda_G$ is a continuous and piecewise continuously differentiable
function ($g$ denotes the inradius throughout). We will denote ${\rm Vol} (G)$ again by $-\kappa_d(G)$.

For additional simplicity we also assume that there is a single
generator $G$. ( In the case of multiple generators we have to
apply the formula to each generator separately and add them up.)
\begin{figure}[h]
\begin{center}
\begin{pspicture}(0,0)(10,5)
\psset{unit=0.9cm}
\psline[linewidth=0.7pt](0,2)(3,0)(7,1.5)(10,0)(8,3)(9,5)(5,3)(3,5)(2.5,2.5)(0,2)
\rput(5.2,0.5){$a_1$} \rput(8.2,0.55){$a_2$} \rput(9,2){$a_3$}
\uput[d](3,0){$A_1$}\uput[d](7,1.5){$A_2$}\uput[d](10,0){$A_3$}\uput[l](0,2){$A_n$}
\rput(1.3,2.6){$a_{n-1}$} \rput(1.3,0.7){$a_{n}$}

\psarc(3,0){0.5}{19}{145} \uput[u](3,0.5){$\alpha_1$}
\psarc(7,1.5){0.3}{332}{200} \rput[u](7,2.1){$\alpha_2$}
\psarc(10,0){0.7}{125}{152} \rput[u](9.2,0.75){$\alpha_3$}
\psarc(2.5,2.5){0.3}{189}{80} \rput[r](3.9,2.4){$\alpha_{n-1}$}
\psarc(0,2){0.7}{327}{11} \rput(1,1.8){$\alpha_{n}$}
\end{pspicture}
\caption{}\label{Fig:polygon}
\end{center}
\end{figure}
\begin{remark}
For any (not self-intersecting) polygon $P$ in the plane, the
volume of the inner $\varepsilon$-neighborhood of $P$ is given by
a quadratic polynomial
 in $\varepsilon$ for sufficiently small $\varepsilon$. Let the polygon $P$ has
the vertices  $A_1,A_2,\dots,A_n$ ( with $A_{n+1}=A_1$) in a
certain successive ordering. Denote the length of the edge
$A_iA_{i+1}$ with $a_i$ and the inner angle at $A_i$ by
$\alpha_i,\,\, (i=1,2,\dots,n)$, see Fig.\ref{Fig:polygon}. Then
it can be shown that for small $\varepsilon$
\[V_P(\varepsilon)=  \left(\sum_{i=1}^n a_i \right)\varepsilon -
\left(\sum_{i=1}^n  \delta_i\right)\varepsilon^2,\] where
\[\delta_i=\left\{
  \begin{array}{cc}
   \displaystyle\frac{1+\cos(\alpha_i)}{\sin{\alpha_i}} &\text{ for } 0<\alpha_i<\pi \\
   \displaystyle\frac{\pi-\alpha_i}{2} &\text{ for } \pi \leq \alpha <
   2\pi .
 \end{array}
 \right . \]
\end{remark}

\begin{figure}[h]
\begin{center}
\includegraphics[height=7cm]{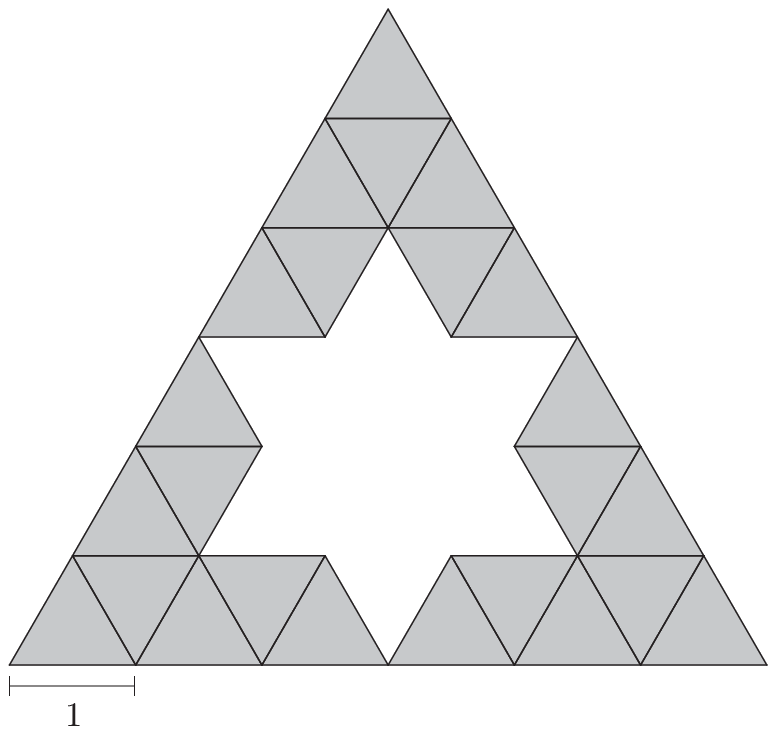}
\end{center}
\caption{}\label{Fig:maps}
\end{figure}
\begin{figure}[h]
\begin{center}
\includegraphics[height=7cm]{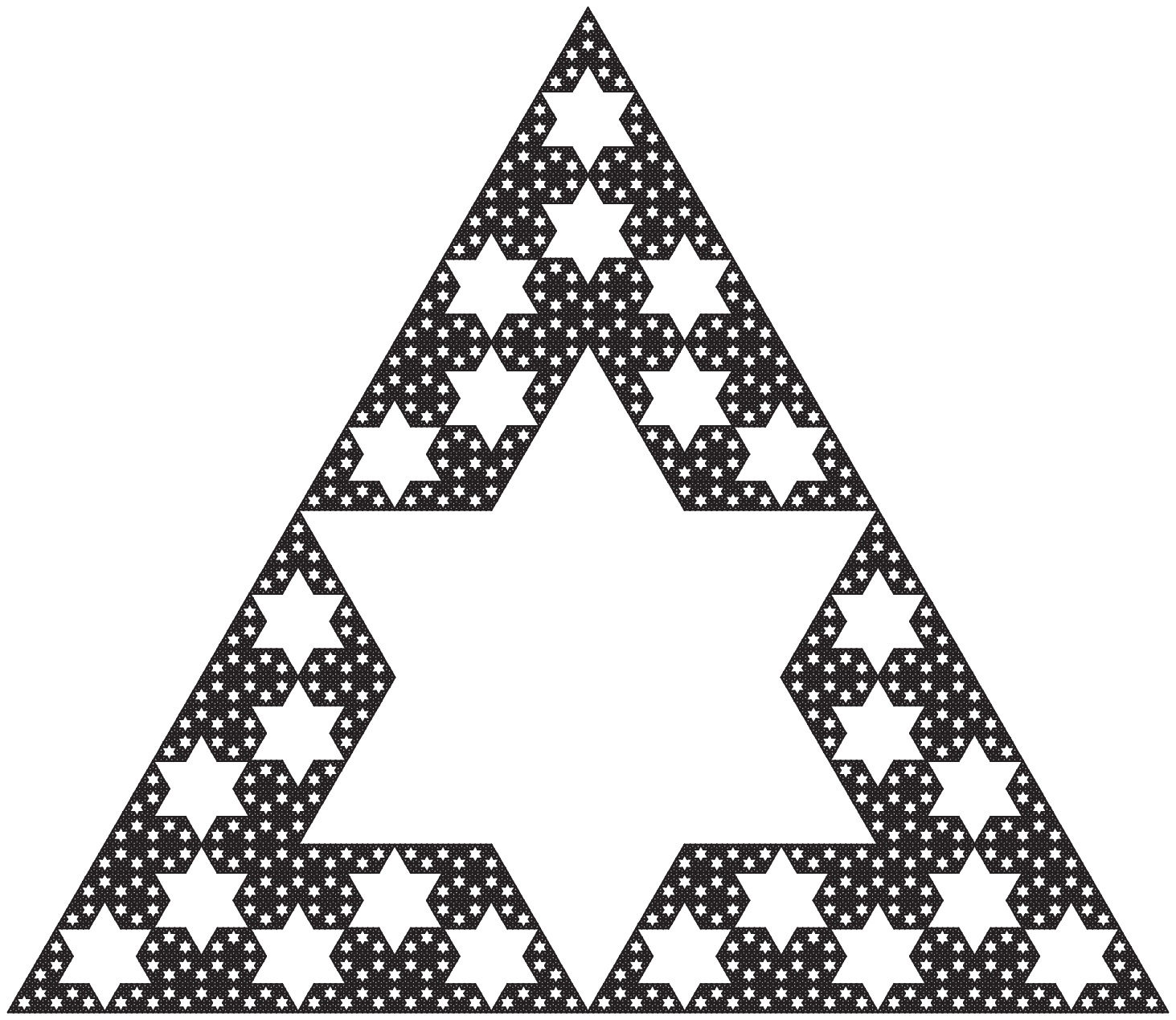}
\end{center}
\caption{}\label{Fig:fractal}
\end{figure}

\begin{example}
Consider the iterated function system $\Phi=\{\phi_j\}_{j=1}^{24}$ on $\mathbb{R}^2$ with scaling ratios $\frac{1}{6}$ as indicated in Fig.\ref{Fig:maps}. The associated self-similar fractal is shown in Fig.\ref{Fig:fractal}. This system satisfies the Pearse-Winter condition and the generator is a non-convex polygon $P$ ( with
$a_i=1,\,i=1,2,\dots,24;\,\alpha_i =\frac{\pi}{3}$ for odd $i$ and
$\alpha_i=\frac{4\pi}{3}$ for even $i$), see Fig.\ref{Fig:star}.

By the formula above we have $V_P(\varepsilon)=12
\varepsilon-(6\sqrt{3}-\pi)\varepsilon^2 \text{ for small }
\varepsilon$. More precisely one can compute\[
V_G(\varepsilon)=\left \{ \begin{array}{cc} 12
\varepsilon-(6\sqrt{3}-\pi)\varepsilon^2  & \text{ for } 0<\varepsilon < \frac{1}{\sqrt{3}} \\
\lambda_P(\varepsilon)
&\text{ for } \frac{1}{\sqrt{3}}\leq \varepsilon \leq 1, \\
\end{array}
\right .\]
where $\lambda_P(\varepsilon)=\frac{6\sqrt{3}+3\sqrt{16\varepsilon^2-4}}{4}+6\varepsilon^2
\arcsin\left(\frac{2\sqrt{3}-\sqrt{16\varepsilon^2-4}}{8\varepsilon}\right)$. (Note that $\kappa_0(P)=\pi-6\sqrt{3}, \, \kappa_1(P)=12$ and $\kappa_2(P)=-3 \sqrt{3}$.)
\end{example}

\begin{figure}[h]
\begin{center}
\begin{pspicture}(-2,-0.5)(2,5)
\psset{unit=1.3cm} 
\SpecialCoor
\def \ok{
\psline[linewidth=0.5pt](0,0)(1;60)
\psline[linewidth=0.5pt](0,0)(1;120)
\qdisk(0,0){1.2pt} \qdisk(1;60){1.2pt}  }  \ok \rput{60}(1.73;30){\ok}
\rput{120}(3;60){\ok} \rput{180}(3.46;90){\ok}
\rput{240}(3;120){\ok} \rput{300}(1.73;150){\ok}
\psarc[linewidth=0.5pt](0,0){0.4}{60}{120}
\psarc[linewidth=0.5pt](1;60){0.2}{0}{240}
\uput[d](0,0){$A_1$}\uput[dr](1.73;30){$A_3$}\uput[ur](3;60){$A_5$}
\uput[u](3.46;90){$A_7$}\uput[ul](3;120){$A_9$}\uput[dl](1.73;150){$A_{11}$}
\uput[dr](1;60){$A_2$}\uput[r](2;60){$A_4$}
\uput[dl](0.98;115){$A_{12}$}\uput[l](2;120){$A_{10}$}
\uput[ur](0.4,2.65){$A_6$}\uput[ul](-0.4,2.65){$A_8$}
\rput(0,0.6){$\frac{\pi}{3}$}\rput(0.5,1.3){$\frac{4\pi}{3}$}

\end{pspicture}
\caption{}\label{Fig:star}
\end{center}
\end{figure}
We still use the notion of scaling $\zeta$-function and the
associated complex dimensions in the sense of L-P.  Since the
geometric $\zeta$-function depends on the geometry of the
generators, we want to define a new geometric $\zeta$-function
taking into account the type of generators satisfying (\ref{vegeepsilon}).

\begin{definition}
The geometric $\zeta$-function
$\zeta_{\mathcal{T}}(s,\varepsilon)$ associated with the generator
$G$ satisfying (\ref{vegeepsilon}) is defined by
\begin{equation*}
\zeta_{\mathcal{T}}(s,\varepsilon):=\zeta(s)\varepsilon^{d-s}
\left ( \sum_{i=0}^{d-1}
\frac{h^{s-i}}{s-i}\kappa_i(G)+\frac{g^{s-d}}{s-d}\kappa_d(G)+\Lambda(s)
\right ),
\end{equation*}
where $\Lambda(s)$ is an entire function given by  \[\Lambda(s)=\int_h^g u^{s-d-1}\lambda_G(u)du.\]
\end{definition}
Now we can state our version of the L-P formula for generators
satisfying the condition (\ref{vegeepsilon}):

\begin{theorem}\label{bizimteo}
\begin{equation*}
V_{\mathcal{T}} (\varepsilon)=\sum_{\omega \in
\mathfrak{D}_{\mathcal{T}}} {\rm{res}}
(\zeta_{\mathcal{T}}(s,\varepsilon);\omega) \quad \text{ for }\varepsilon < h,
\end{equation*}
where $\mathfrak{D}_{\mathcal{T}}=\mathfrak{D} \cup
\{0,1,\dots,d-1\}$.
\end{theorem}

\begin{remark}
Note that $d$ is not included in the set $\mathfrak{D}_{\mathcal{T}}$. The reason of this exclusion will be clear from the proof of the theorem.
\end{remark}

\setcounter{example}{0}
\begin{example}[continued]

The scaling $\zeta$-function of the system $\Phi$ of Example 1 is given by \[\zeta(s)=\frac{1}{1-\displaystyle \sum_{j=1}^{24}r_j^s}=\frac{1}{1-24(\frac{1}{6})^s}\quad \text{for } \rm{Re(s)>D},\] where $D=1+\log_6{4}$.

The set of complex dimension is $\mathfrak{D}=\{D+ \mathbbm{i}np \mid n \in \mathbb{Z} \}$ where $p=\frac{2\pi}{\log 6}$.

The geometric $\zeta$-function is given by
\begin{equation*}
\zeta_{\mathcal{T}}(s,\varepsilon)=\zeta(s)\varepsilon^{2-s}\left ( \displaystyle \frac{\left (\frac{1}{\sqrt{3}}\right )^s}{s}(\pi-6\sqrt{3})+ \frac{\left (\frac{1}{\sqrt{3}}\right )^{s-1}}{s-1} 12+\frac{1}{s-2}(-3\sqrt{3}) +  \displaystyle \int_{\frac {1}{\sqrt{3}}}^1 u^{s-3} \lambda_G(u)du \right ).
\end{equation*}

Thus by Theorem \ref{bizimteo} we get
\begin{eqnarray*}
V_{\mathcal{T}}
(\varepsilon)&=&\sum_{\omega \in \mathfrak{D} \cup \{0,1\}}{\rm{res}}(\zeta_{\mathcal{T}}(s,\varepsilon);\omega)\\
&=&\sum_{n \in \mathbb{Z}} {\rm{res}} (\zeta_{\mathcal{T}}(s,\varepsilon);D+\mathbbm{i}np)+ {\rm{res}} (\zeta_{\mathcal{T}}(s,\varepsilon);0)+{\rm{res}} (\zeta_{\mathcal{T}}(s,\varepsilon);1)\\
&=& \varepsilon^{2-D-\mathbbm{i}np} \frac{1}{\log 6} \left( \frac{  \left ( \frac{1}{\sqrt{3}}\right )^{D+\mathbbm{i}np}}{D+\mathbbm{i}np}(\pi-6\sqrt{3}) +  \frac{\left (\frac{1}{\sqrt{3}}\right )^{D-1+\mathbbm{i}np}}{D-1+\mathbbm{i}np} 12 +\frac{3\sqrt{3}}{D-2+\mathbbm{i}np} + \Lambda(D+\mathbbm{i}np)\right ) \\
&\,&+\frac{6\sqrt{3}-\pi}{23}\, \varepsilon^2-4\varepsilon.
\end{eqnarray*}
\end{example}

\section{Proof of Theorem \ref{bizimteo}}

Our goal is to find a closed expression for $V_{\mathcal{T}}
(\varepsilon)=\displaystyle \sum_{n=1}^{\infty}V_{T_n}(\varepsilon)$ as stated in Theorem \ref{bizimteo}. As we  assumed a single generator for simplicity, the volume of the inner $\varepsilon$-tube of the tiling $\mathcal{T}$ is, by the tileset condition, the sum of the volumes of the inner $\varepsilon$-neighborhoods of all the scaled copies of $G$ appearing in the tiling:
\[V_{\mathcal{T}}
(\varepsilon)=\displaystyle \sum_{k=1}^{\infty} \sum_{w \in W_k} V_{r_w G}(\varepsilon),\] where $W_k$ and $r_w$ are as in Definition \ref{scalingzeta} and $r_w G$ is a copy of $G$ scaled by $r_w$. A simple calculation shows that, if $V_G(\varepsilon)$ is given as in (\ref{vegeepsilon}), then
\[V_{xG}(\varepsilon)= \left \{
\begin{array}{cl}
  \displaystyle \sum_{j=0}^{d-1}\kappa_j(G)x^j\varepsilon^{d-j} &\text{ for } \varepsilon < xh \\
  x^d \lambda (\frac{\varepsilon}{x}) &\text{ for } xh \leq \varepsilon \leq xg \\
  -x^d \kappa_d(G) &\text{ for } \varepsilon > xg.
\end{array}
 \right .
\]
It will be more convenient for us to regard $V_{xG}(\varepsilon)$ as a two-variable function of $x$ and $\varepsilon$:
\[V_{G}(x,\varepsilon):= V_{xG}(\varepsilon)=\left \{
\begin{array}{cl}
  \displaystyle -x^d \kappa_d(G) &\text{ for } 0 <x < \frac{\varepsilon}{g} \\
  x^d \lambda (\frac{\varepsilon}{x}) &\text{ for } \frac{\varepsilon}{g} \leq x \leq \frac{\varepsilon}{h} \\
  \displaystyle \sum_{j=0}^{d-1}\kappa_j(G)x^j\varepsilon^{d-j} &\text{ for } x > \frac{\varepsilon}{h}.
\end{array}
 \right .
\]

Recall that the Mellin transform $\mathcal{M}[f;s]$ of a function $f:(0,\infty)\rightarrow \mathbb{R}$ is given by \[\mathcal{M}[f;s]=\widetilde{f}(s)=\int_0^{\infty}x^{s-1}f(x)dx.\]

For fixed $\varepsilon$, we take the Mellin transform of $V_G(x,\varepsilon)$ as a function of $x$:
\begin{eqnarray*}
\widetilde{V}_G(s,\varepsilon)&=&\int_0^{\frac{\varepsilon}{g}} x^{s-1} (-x^d\kappa_d(G))dx + \int_{\frac{\varepsilon}{g}}^{\frac{\varepsilon}{h}} x^{s-1} x^d\lambda (\frac{\varepsilon}{x} ) dx + \int_{\frac{\varepsilon}{h}}^{\infty} x^{s-1} \left( \sum_{j=0}^{d-1} \kappa_j(G) x^j \varepsilon^{d-j} \right ) dx \\
&=& -\varepsilon^{s+d}\left( \frac{\kappa_d(G)}{g^{s+d}(s+d)}+ \sum_{j=0}^{d-1} \frac{\kappa_j(G)}{h^{s+j}(s+j)} \right )+\int_h^g \frac{\varepsilon^{s+d}}{u^{s+d+1}}\lambda(u)du,
\end{eqnarray*}
for $-d< {\rm Re}(s) < 1-d $. (In the second integral we made the change of variable $\frac{\varepsilon}{x} \to u $.)

\begin{eqnarray}
\widetilde{V}_G(s,\varepsilon)&=& -\varepsilon^{s+d}\left( \frac{\kappa_d(G)}{g^{s+d}(s+d)}+ \sum_{j=0}^{d-1} \frac{\kappa_j(G)}{h^{s+j}(s+j)}-\int_h^g u^{-s-d-1}\lambda(u)du \right ) \nonumber\\
\widetilde{V}_G(s,\varepsilon)&=& -\varepsilon^{s+d}\left( \frac{\kappa_d(G)}{g^{s+d}(s+d)}+ \sum_{j=0}^{d-1} \frac{\kappa_j(G)}{h^{s+j}(s+j)}-\Lambda(-s) \right ), \label{vgtildeepsilon}
\end{eqnarray}
where $\Lambda(s)=\displaystyle \int_h^g{u^{s-d-1}}\lambda(u)du$.

If we take the inverse Mellin transform of $\widetilde{V}_G(s,\varepsilon)$, we obtain \[V_G(x,\varepsilon)=\mathcal{M}^{-1}\left[\widetilde{V}_G(s,\varepsilon);x  \right] =\frac{1}{2\pi \mathbbm{i}}  \displaystyle
\int_{-c-\mathbbm{i}\infty}^{-c+\mathbbm{i}\infty}
x^{-s}\widetilde{V}_G(s,\varepsilon)ds, \] where $d-1<c<d$. For an additional purpose below, we choose $c$ such that $d-1<D<c<d$.

We shall now insert the above expression for $V_G(x,\varepsilon)$ into the sum
\[V_{\mathcal{T}}(\varepsilon)=\sum_{k=0}^\infty \sum_{w \in W_k} V_{r_w G}(\varepsilon),\]
but for the ease of the computation, we order the scaling coefficients $r_w$ into a sequence $\{x_m\}_{m=1}^{\infty}$: \[V_{\mathcal{T}}(\varepsilon):=\sum_{m=1}^{\infty}
V_{x_m G}(\varepsilon),
\] so that  we get
\begin{equation} \label{nebir}
V_{\mathcal{T}}(\varepsilon)=\sum_{m=1}^{\infty}\frac{1}{2\pi \mathbbm{i}} \int_{-c-\mathbbm{i}\infty}^{-c+\mathbbm{i}\infty} x_m^{-s} \widetilde{V}_G(s,\varepsilon)ds.
\end{equation}

We can now change the order of the sum and the integral. (For justification see the appendix.)
\begin{eqnarray*}
V_{\mathcal{T}}(\varepsilon)&=&\frac{1}{2\pi \mathbbm{i}} \int_{-c-\mathbbm{i}\infty}^{-c+\mathbbm{i}\infty}\sum_{m=1}^{\infty}x_m^{-s} \widetilde{V}_G(s,\varepsilon)ds\\
&=&\frac{1}{2\pi \mathbbm{i}} \int_{-c-\mathbbm{i}\infty}^{-c+\mathbbm{i}\infty}\zeta(-s) \widetilde{V}_G(s,\varepsilon)ds\quad \text{ by Def.\ref{scalingzeta}}.
\end{eqnarray*}
Changing the variable of the integral by $s'=-s$ we find \[V_{\mathcal{T}}(\varepsilon)=\frac{1}{2\pi \mathbbm{i}} \int_{c-\mathbbm{i}\infty}^{c+\mathbbm{i}\infty}\zeta(s) \widetilde{V}_G(-s,\varepsilon)ds.\]
By (\ref{vgtildeepsilon}) we get
\begin{eqnarray*}
V_{\mathcal{T}}(\varepsilon)&=&\frac{1}{2\pi \mathbbm{i}} \int_{c-\mathbbm{i}\infty}^{c+\mathbbm{i}\infty}\zeta(s) \varepsilon^{d-s}\left( \frac{\kappa_d(G)}{g^{d-s}(s-d)}+ \sum_{j=0}^{d-1} \frac{\kappa_j(G)}{h^{j-s}(s-j)}+\Lambda(s) \right )ds\\
&=&\frac{1}{2\pi \mathbbm{i}} \int_{c-\mathbbm{i}\infty}^{c+\mathbbm{i}\infty}\zeta_{\mathcal{T}}(s,\varepsilon) ds.
\end{eqnarray*}

The poles of $\zeta(s)$ are contained in a horizontally bounded strip $D_{\ell} \leq {\rm Re}(s) \leq D$, for some real number $D_{\ell}$, which can be assumed to be negative also (see \cite[Theorem. 2.17]{LaFra}).
Let $c_{\ell}$ be chosen such that
\begin{eqnarray}
& &\text{i)}\quad c_{\ell}<D_{\ell} \\
& &\text{ii)}\quad \zeta(s) \text{ is bounded for } {\rm Re}(s)\leq c_{\ell}. \label{uci}
\end{eqnarray}
(The second property is possible because $|\zeta(s)| \to 0$ as ${\rm Re}(s)\to -\infty$.)

\begin{figure}[h]
\begin{center}
\begin{pspicture}(-7,-2)(4,2)
\pspolygon[linewidth=0.01pt,fillstyle=solid,fillcolor=lightgray](1.5,-2)(1.5,2)(-5,2)(-5,-2)
\psline [linewidth=0.5pt](-7,0)(-3,0) \psline
[linewidth=1.1pt,linestyle=dotted,dotsep=2pt](-2.5,0)(-3,0)
\psline[linewidth=0.5pt,arrowsize=4pt]{->}(-2.5,0)(4,0)
\psline[linewidth=0.5pt](2.3,-2)(2.3,2)\psline[linewidth=0.5pt](1.5,-2)(1.5,2)
\psline[linewidth=0.5pt,arrowsize=4pt]{->}(2.3,-2)(2.3,1)
\psline[linewidth=0.5pt](-6,-2)(-6,2)
\psline[linewidth=0.5pt,arrowsize=4pt]{->}(-6,-2)(-6,1)
\psline[linewidth=0.5pt](-5,-2)(-5,2)
\uput[d](3.4,0){$d$}\uput[d](0.8,0){ \small$
d-1$}\uput[d](-1.8,0){\small$d-2$} \uput[ur](1.45,0){$D$}
\uput[ul](-6,0){$c_{\ell}$}\uput[ur](2.3,0){$c$}\uput[d](-4,0){$0$}
\uput[ul](-5,0){$D_{\ell}$} \uput[r](2.3,-1.8){$c-\mathbbm{i}
\infty$}\uput[r](2.3,1.8){$c+\mathbbm{i}
\infty$}\uput[l](-6,-1.8){$c_{\ell}-\mathbbm{i}
\infty$}\uput[l](-6,1.8){$c_{\ell}+\mathbbm{i} \infty$}
\qdisk(-6,0){1.5pt}\qdisk(-5,0){1.5pt}\qdisk(-4,0){1.5pt}\qdisk(1.5,0){1.5pt}
\qdisk(-1.8,0){1.5pt}\qdisk(0.8,0){1.5pt}\qdisk(2.3,0){1.5pt}\qdisk(3.4,0){1.5pt}
\end{pspicture}
\end{center}
\caption{}\label{Fig:strip}
\end{figure}

We now proceed to show that (see Fig.\ref{Fig:strip})
\begin{equation}\label{ilksart}
\frac{1}{2\pi \mathbbm{i}} \int_{c-\mathbbm{i}\infty}^{c+\mathbbm{i}\infty}\zeta_{\mathcal{T}}(s,\varepsilon) ds=\frac{1}{2\pi \mathbbm{i}} \int_{c_{\ell}-\mathbbm{i}\infty}^{c_{\ell}+\mathbbm{i}\infty}\zeta_{\mathcal{T}}(s,\varepsilon) ds+ \sum_{\omega \in \mathfrak{D}\cup \{0,1,\dots,d-1\} }{\rm res}(\zeta_{\tau}(s,\varepsilon);\omega)
\end{equation}
and
\begin{equation}\label{ikincisart}
\frac{1}{2\pi \mathbbm{i}} \int_{c_{\ell}-\mathbbm{i}\infty}^{c_{\ell}+\mathbbm{i}\infty}\zeta_{\mathcal{T}}(s,\varepsilon) ds = 0 \quad \text{ for } \varepsilon < h.
\end{equation}

Clearly, this will complete the proof of the theorem.

We begin with (\ref{ilksart}): First note that there exists a positive increasing sequence $\{\tau_n\}_{n=1}^{\infty}$ with
$\{\tau_n\} \to \infty $ such that
\begin{equation}\label{n6}
|\zeta(\sigma\pm \mathbbm{i} \tau_n)| \leq M \text{ for  } c_{\ell}\leq \sigma \leq d \text { and for all }  n,
\end{equation}
where $M$ is some positive constant. (see \cite[Theorem 3.26]{LaFra}.)

Denote the rectangle with the corners
$c-\mathbbm{i}\tau_n,c+\mathbbm{i}\tau_n,c_{\ell}+\mathbbm{i}\tau_n,c_{\ell}-\mathbbm{i}\tau_n$
(and with the edges $L_{1,n},L_{2,n},L_{3,n},L_{4,n}$) by $S_n$
(see Fig. \ref{Fig:rectangle}). By residue theorem \[ \frac{1}{2 \pi
\mathbbm{i}}\int _{\partial S_n}
\zeta_{\mathcal{T}}(s,\varepsilon)ds= \sum_{\omega \in \left (
\mathfrak{D} \cup \{0,1,\dots,d-1\}  \right
) \cap S_n} {\rm {res}} (\zeta_{\mathcal{T}}(s,\varepsilon);\omega).
\]
(Note that $d$ lies outside the rectangles $S_n$ and does not contribute to the above sum.)

\begin{figure}[h]
\begin{center}
\begin{pspicture}(-3,-2)(4,2.5)
\psline [linewidth=0.5pt,arrowsize=4pt]{->}(-2,0)(3.5,0)
\psline[linewidth=0.5pt](2,-2)(2,2)
\psline[linewidth=0.5pt,arrowsize=4pt]{->}(2,-2)(2,1)
\psline[linewidth=0.5pt](-1,-2)(-1,2)
\psline[linewidth=0.5pt,arrowsize=4pt]{<-}(-1,-1)(-1,1)
\psline[linewidth=0.5pt](-1,2)(2,2)
\psline[linewidth=0.5pt](-1,-2)(2,-2)
\psline[linewidth=0.5pt,arrowsize=4pt]{<-}(0.5,2)(2,2)
\psline[linewidth=0.5pt,arrowsize=4pt]{->}(-1,-2)(0.5,-2)
\uput[r](2,1){$L_{1,n}$} \uput[u](0.5,2){$L_{2,n}$}
\uput[l](-1,-1){$L_{3,n}$}\uput[d](0.5,-2){$L_{4,n}$}
\uput[ul](-1,0){$c_{\ell}$}\uput[ur](2,0){$c$}
\uput[r](2,-2){$c-\mathbbm{i} \tau_n$}\uput[r](2,2){$c+\mathbbm{i}
\tau_n$} \uput[l](-1,-2){$c_{\ell}-\mathbbm{i}
\tau_n$}\uput[l](-1,2){$c_{\ell}+\mathbbm{i}
\tau_n$}\rput(0.5,0.8){$S_n$}\qdisk(-1,0){1.5pt}\qdisk(2,0){1.5pt}\qdisk(2.9,0){1.5pt}
\qdisk(2,-2){1pt}\qdisk(2,2){1pt}\qdisk(-1,-2){1pt}\qdisk(-1,2){1pt}\uput[u](2.9,0){$d$}
\end{pspicture}
\end{center}
\caption{The rectangle $S_n$.} \label{Fig:rectangle}
\end{figure}
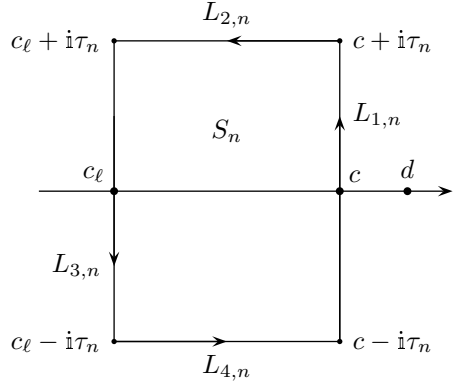

Therefore
 \begin{eqnarray}
\frac{1}{2 \pi \mathbbm{i}} \left ( \int_{L_{1,n}}
\zeta_{\mathcal{T}}(s,\varepsilon)ds +\int_{L_{2,n}}
\zeta_{\mathcal{T}}(s,\varepsilon)ds +\int_{L_{3,n}}
\zeta_{\mathcal{T}}(s,\varepsilon)ds +\int_{L_{4,n}}
\zeta_{\mathcal{T}}(s,\varepsilon)ds \right ) \label{inttop}\\
=\sum_{\omega \in \mathfrak{D} \cup
\{0,1,\dots,d-1\} \cap S_n} {\rm {res}}
(\zeta_{\mathcal{T}} (s,\varepsilon);\omega). \nonumber
\end{eqnarray}

Let us consider first the integral over $L_{2,n}$. For $s \in L_{2,n}$ we have $s=\sigma+ \mathbbm{i} \tau_n$, $c_{\ell}\leq \sigma \leq c$.
\begin{eqnarray*}
|\zeta_{\mathcal{T}} (s,\varepsilon)|&=& \left \vert \zeta(s) \varepsilon^{d-s} \left(  \sum_{j=0}^{d-1} \frac{h^{s-j}}{(s-j)}\kappa_j(G)+\frac{g^{s-d}}{(s-d)}\kappa_d(G)+\Lambda(s) \right ) \right \vert \\
&\leq& M \varepsilon^{d-\sigma} \left(  \sum_{j=0}^{d-1} \frac{h^{\sigma-j}}{|\sigma+\mathbbm{i}\tau_n-j|} |\kappa_j(G)|+\frac{g^{\sigma-d}}{|\sigma+\mathbbm{i}\tau_n-d|}|\kappa_d(G)|+|\Lambda(s)| \right )\quad \text{by (\ref{n6})}\\
&\leq& M \varepsilon^{d-\sigma} \left(  \sum_{j=0}^{d-1} \frac{h^{\sigma-j}|\kappa_j(G)|}{\tau_n}+\frac{g^{\sigma-d}|\kappa_d(G)|}{\tau_n} \right ) + M \varepsilon^{d-\sigma} |\Lambda(s)| \\
&\leq&  \frac{MM'}{\tau_n} + M \varepsilon^{d-\sigma} |\Lambda(s)|,
\end{eqnarray*}
where $M'=\underset{c_{\ell}\leq \sigma \leq c}{\max} \left\{ \varepsilon^{d-\sigma} \left( \sum_{j=0}^{d-1} h^{\sigma-j}|\kappa_j(G)|+g^{\sigma-d} |\kappa_d(G)| \right ) \right \}$.

Now we consider the term $M \varepsilon^{d-\sigma} |\Lambda(s)|$: Recall that
\[ \Lambda (s) = \int_{h}^{g} u^{s-d-1} \lambda (u)du.\] Integrating by parts,
\begin{equation} \label{nussu}
\Lambda (s) = \left . \frac{u^{s-d}}{s-d} \lambda (u) \right \vert_{h}^{g}- \int_{h}^{g} \frac{u^{s-d}}{s-d} \lambda' (u)du,
\end{equation}
whence we can write $\varepsilon^{d-\sigma} |\Lambda(s)| \leq \frac{M''}{|s-d|} \leq \frac{M''}{\tau_n}$ by the assumption of continuity and piecewise continuous differentiability of $\lambda(u)$ and $c_{\ell}\leq \sigma \leq c$. Thus we obtain \[ \left \vert \zeta_{\mathcal{T}} (s,\varepsilon) \right \vert \leq \frac{MM'}{\tau_n}+\frac{MM''}{\tau_n}.\] Therefore, \[\left \vert \int_{L_{2,n}}\zeta_{\mathcal{T}} (s,\varepsilon) \right \vert \leq \frac{M(M'+M'')}{\tau_n}(c-c_{\ell}) \to 0 \text{ for } n \to \infty,\] since $\tau_n \to \infty$ as $n \to \infty$.

Similarly the integral over $L_{4,n}$ tends to $0$ as $n \to \infty $. Thus letting $n\to \infty$ in (\ref{inttop}) gives (\ref{ilksart}).

Now, we will show (\ref{ikincisart}): Let $L_n$ be the line
segment $L_n(t)=c_{\ell}+\mathbbm{i}t, \, -n \leq t \leq n;$ $C_n$
be the semicircle $C_n(t)=c_{\ell}+n e^{\mathbbm{i}t}, \,
\frac{\pi}{2} \leq t \leq \frac{3\pi}{2}$ and $\Gamma_n=L_n+C_n$
(see Fig.\ref{Fig:semi}). By the choice of $c_{\ell}$,
$\zeta_{\mathcal{T}}(s,\varepsilon)$ is analytic on and inside
$\Gamma_n$ and therefore \[ \int_{L_n}
\zeta_{\mathcal{T}}(s,\varepsilon)ds=-\int_{C_n}
\zeta_{\mathcal{T}}(s,\varepsilon)ds.\]

\begin{figure}[h]
\begin{center}
\begin{pspicture}(-2,-2)(4,2)
\psline [linewidth=0.5pt,arrowsize=4pt]{->}(-1,0)(4,0)
\psline[linewidth=0.5pt](2,-2)(2,2)
\psline[linewidth=0.5pt,arrowsize=4pt]{->}(2,-2)(2,1)
\psarc[linewidth=0.5pt](2,0){2}{90}{270}
\psarc[linewidth=0.5pt,arrowsize=4pt]{->}(2,0){2}{90}{150}
 \uput[r](2,1){$L_{n}$}
\uput[ur](2,0){$c_{\ell}$}
\uput[r](2,-2){$c_{\ell}-\mathbbm{i} n$}\uput[r](2,2){$c_{\ell}+\mathbbm{i} n$}\rput(0,1.5){$C_n$}\qdisk(2,2){1pt}\qdisk(2,-2){1pt}
\end{pspicture}
\end{center}
\caption{The contour $\Gamma_n=L_n+C_n$.}\label{Fig:semi}
\end{figure}
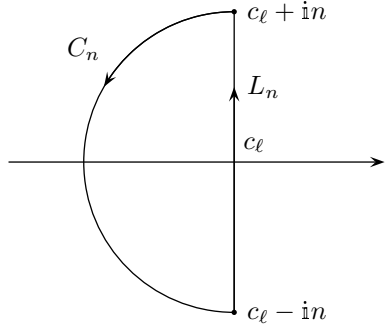

We will be done when we show that the right-hand side tends to $0$
as $n \to \infty$. Using the parametrization of $C_n$, we obtain
\begin{eqnarray*}
\left \vert \int_{C_n} \zeta_{\mathcal{T}}(s,\varepsilon)ds \right
\vert &=& \left \vert \int_{\frac{\pi}{2}}^{ \frac{3\pi}{2}}
\zeta(c_{\ell}+n e^{\mathbbm{i}t }) \varepsilon^{d-c_{\ell}-n
e^{\mathbbm{i}t }} \left(   \sum_{j=0}^{d-1} \frac{ h^{c_{\ell}+n
e^{\mathbbm{i}t }-j} }{ c_{\ell}+n e^{\mathbbm{i}t }-j
}\kappa_j(G) \right. \right.\\
& & + \left. \left.    \frac{ g^{c_{\ell}+n e^{\mathbbm{i}t }-d}
}{ c_{\ell}+n e^{\mathbbm{i}t }-d }\kappa_d(G) +\Lambda(c_{\ell}+n
e^{\mathbbm{i}t })     \right) \mathbbm{i} n e^{\mathbbm{i}t} dt
\right \vert.
\end{eqnarray*}

By condition (\ref{uci}), $\zeta(s)$ is bounded for ${\rm Re}(s) \leq
c_{\ell}$, say $\left \vert \zeta(s)\right \vert \leq K$:
\begin{eqnarray*}
\left \vert \int_{C_n} \zeta_{\mathcal{T}}(s,\varepsilon)ds \right
\vert & \leq & K \varepsilon^{d-c_{\ell}}  \left( \sum_{j=0}^{d-1}
\frac{n\, h^{c_{\ell}-j}  }{ n- \vert c_{\ell}-j \vert } \vert
\kappa_j(G) \vert    \int_{\frac{\pi}{2}}^{ \frac{3\pi}{2}} \left(
\frac{h}{\varepsilon}\right)^{n \cos t } dt \right.\\
& &+ \frac{n\, g^{c_{\ell}-d} }{ n- \vert c_{\ell}-d \vert } \vert
\kappa_d(G) \vert  \int_{\frac{\pi}{2}}^{ \frac{3\pi}{2}} \left(
\frac{g}{\varepsilon}\right)^{n \cos t } dt\\
& & \left.+ n  \int_{\frac{\pi}{2}}^{ \frac{3\pi}{2}} \left \vert
\Lambda( c_{\ell}+n e^{\mathbbm{i}t } )\right \vert
\varepsilon^{-n \cos t} dt \right).
\end{eqnarray*}
Let us denote the right-hand side by $I_1+I_2+I_3$. We have \[I_1=K \varepsilon^{d-c_{\ell}}  \sum_{j=0}^{d-1}
\frac{n\, h^{c_{\ell}-j}  }{ n- \vert c_{\ell}-j \vert } \vert
\kappa_j(G) \vert    \int_{\frac{\pi}{2}}^{ \frac{3\pi}{2}} \left(
\frac{h}{\varepsilon}\right)^{n \cos t } dt \leq K_1 \int_{\frac{\pi}{2}}^{ \frac{3\pi}{2}} \left(
\frac{h}{\varepsilon}\right)^{n \cos t } dt,\] where $K_1$ is some constant.

The well-known Jordan Lemma \cite{churchill} states that \[\lim_{n\to\infty} \int_{\frac{\pi}{2}}^{ \frac{3\pi}{2}} a^{n \cos t } dt=0,\] for any fixed $a>1$.
Since  we assumed $\varepsilon < h$ for Theorem \ref{bizimteo}, we thus get $I_1 \to 0$ as $n \to \infty$.

Similarly $I_2 \to 0$ as $n\to \infty$, since $\varepsilon <g$ as $h<g$.

To deal with the term $I_3$, recall that by (\ref{nussu}) \[\Lambda(s)=\left .\frac{u^{s-d}}{s-d} \lambda (u) \right \vert_h^g - \int_{h}^{g} \frac{u^{s-d}}{s-d} \lambda'(u)du. \]
By continuity and piecewise continuous differentiability of $\lambda$, we can write \[
|\Lambda (s)| \leq \frac {1}{|s-d|} \left ( K_1\, g^{{\rm Re}(s)-d}+ K_2\, h^{{\rm Re}(s)-d} \right ) +\frac {1}{|s-d|}\int_h^g u^{{\rm Re}(s)-d}\, K_3\, du
\]
with some constants $K_1,K_2,K_3$.

For $s=c_{\ell}+ne^{\mathbbm{i}t}, (\frac{\pi}{2} \leq t \leq \frac{3\pi}{2})$ we have ${\rm Re}(s)=c_{\ell}+n \cos t <0$ and $u^{{\rm Re}(s)-d}$ is a decreasing function on $[h,g]$.
Thus,
\begin{eqnarray}
|\Lambda (s)| &\leq& K_4 \, \frac {1}{|s-d|}  h^{{\rm Re}(s)-d} +K_3 \, \frac {1}{|s-d|} (g-h)   h^{{\rm Re}(s)-d} \nonumber \\
&\leq &  K_5 \,  \frac{1}{|c_{\ell}+n e^{\mathbbm{i}t}-d|} h^{c_{\ell}+n \cos t-d} \nonumber \\
&\leq & K_6 \, \frac{h^{n \cos t}}{n-|c_{\ell}-d|}. \label{nikiussu}
\end{eqnarray}

Now we show that $I_3 \to 0 $ as $n \to \infty$. Since \[I_3= n \int_{\frac{\pi}{2}}^{\frac{3\pi}{2}} \left \vert \Lambda \left ( c_{\ell}+ne^{\mathbbm{i}t} \right)  \right \vert \varepsilon ^{-n \cos t} dt, \]
using (\ref{nikiussu}) \[I_3 \leq \frac{K_6 \, n}{n-|c_{\ell}-d|} \int_{\frac{\pi}{2}}^{\frac{3\pi}{2}} \left (\frac{h}{\varepsilon} \right )^{n \cos t} dt \to 0 \quad \text{as } n \to \infty, \] by Jordan Lemma. Hence the claim (\ref{ikincisart}) is verified and thus the proof of Theorem \ref{bizimteo} is completed.
\newpage
\begin{appendices}
Our aim here is to justify the change of the orders of the sum and the integral in (\ref{nebir}). That is, we will show that the following equality holds:
\begin{equation}\label{ap_eq_1}
\underset{m=1}{\overset{\infty } {\sum }}\dfrac{1}{2\pi\mathbbm{i}} \int_{c-\mathbbm{i}\infty}^{c+\mathbbm{i}\infty}    x_m^s \widetilde{V}_{G}(-s,\varepsilon) ds
=
\dfrac{1}{2\pi\mathbbm{i}} \int_{c-\mathbbm{i}\infty}^{c+\mathbbm{i}\infty} \underset{m=1}{\overset{\infty } {\sum }} x_m^s \widetilde{V}_{G}(-s,\varepsilon) ds.
\end{equation}

Recall that
\begin{equation}\label{ap_eq_2}
\underset{m=1}{\overset{\infty }{\sum }}x_m^{\alpha}<\infty \text{ \ \ for } \alpha>D
\end{equation}
and we assume $c>D$.

We first prove the following lemma:

\begin{lemma}
For any $R>0$,
\begin{equation}\label{ap_eq_3}
\underset{m=1}{\overset{\infty } {\sum }}  \int_{c-\mathbbm{i}R}^{c+\mathbbm{i}R} x_m^s \widetilde{V}_{G}(-s,\varepsilon) ds
=\int_{c-\mathbbm{i}R}^{c+\mathbbm{i}R} \underset{m=1}{\overset{\infty } {\sum }} x_m^s \widetilde{V}_{G}(-s,\varepsilon) ds.
\end{equation}
\end{lemma}
\begin{proof}
We parameterize the line segment $L_R$ beginning at $c-\mathbbm{i}R$ and ending at $c+\mathbbm{i}R$ by $L_R:L_R(t)=c+\mathbbm{i}t,-R \leq t \leq R$. Then (\ref{ap_eq_3}) is equivalent to
\begin{equation*}
\int_{-R}^R  \underset{m=1}{\overset{\infty } {\sum }} x_m^{c+\mathbbm{i}t} \widetilde{V}_{G}(-c-\mathbbm{i}t,\varepsilon)\mathbbm{i} dt
=\underset{m=1}{\overset{\infty } {\sum }}  \int_{-R}^R x_m^{c+\mathbbm{i}t}\widetilde{V}_{G}(-c-\mathbbm{i}t,\varepsilon) \mathbbm{i} dt.
\end{equation*}
By Fubini's theorem, it suffices to show that
\[
\underset{m=1}{\overset{\infty } {\sum }}  \int_{-R}^R \left\vert x_m^{c+\mathbbm{i}t}\widetilde{V}_{G}(-c-\mathbbm{i}t,\varepsilon) \right\vert  dt<\infty.
\]
Since, $\widetilde{V}_{G}(-s,\varepsilon)$ is analytic in the strip $d-1< \rm{Re(s)}<d$, it is bounded on the line segment $L_R$. So,
\begin{eqnarray*}
\underset{m=1}{\overset{\infty } {\sum }}  \int_{-R}^R \left\vert x_m^{c+\mathbbm{i}t}\widetilde{V}_{G}(-c-\mathbbm{i}t,\varepsilon) \right\vert  dt&\leq& M \underset{m=1}{\overset{\infty } {\sum }}  \int_{-R}^R x_m^c dt\\
&=&2RM \underset{m=1}{\overset{\infty } {\sum }}  x_m^c<\infty,
\end{eqnarray*}
by (\ref{ap_eq_2}).
\end{proof}

Now, right-hand side of (\ref{ap_eq_1}) is
\begin{equation*}
\underset{R\rightarrow \infty}{\lim } \dfrac{1}{2\pi\mathbbm{i}} \int_{c-\mathbbm{i}R}^{c+\mathbbm{i}R} \underset{m=1}{\overset{\infty } {\sum }} x_m^s \widetilde{V}_{G}(-s,\varepsilon) ds
\end{equation*}
and by the above lemma this equals
\begin{equation}\label{ap_eq_4}
\underset{R\rightarrow \infty}{\lim } \dfrac{1}{2\pi\mathbbm{i}} \underset{m=1}{\overset{\infty } {\sum }} \int_{c-\mathbbm{i}R}^{c+\mathbbm{i}R} x_m^s \widetilde{V}_{G}(-s,\varepsilon) ds.
\end{equation}
Let us write
\[
a_{m,R}:=\dfrac{1}{2\pi\mathbbm{i}} \int_{c-\mathbbm{i}R}^{c+\mathbbm{i}R} x_m^s \widetilde{V}_{G}(-s,\varepsilon) ds.
\]
Then (\ref{ap_eq_4}) equals
\begin{equation}\label{ap_eq_5}
\underset{R\rightarrow \infty}{\lim }  \underset{m=1}{\overset{\infty } {\sum }} a_{m,R}.
\end{equation}
Suppose for a moment that $ \exists \{a_m\}_{m=m_0}^{\infty}$ (not depending on $R$) such that for $R>R_0$
\begin{equation}\label{ap_eq_6}
\left\vert a_{m,R}\right\vert\leq a_m \text{\ \ for all } m\geq m_0
\end{equation}
and
\begin{equation}\label{ap_eq_7}
\underset{m=m_0}{\overset{\infty } {\sum }} a_m<\infty.
\end{equation}
Then by Lebesgue dominated convergence theorem, we can interchange the order of the limit and sum in (\ref{ap_eq_5}) and obtain
\begin{eqnarray*}
\underset{R\rightarrow \infty}{\lim }  \underset{m=1}{\overset{\infty } {\sum }} a_{m,R}&=& \underset{m=1}{\overset{\infty } {\sum }} \underset{R\rightarrow \infty}{\lim }   a_{m,R}\\
&=& \underset{m=1}{\overset{\infty } {\sum }}\dfrac{1}{2\pi\mathbbm{i}} \int_{c-\mathbbm{i}\infty}^{c+\mathbbm{i}\infty}    x_m^s \widetilde{V}_{G}(-s,\varepsilon) ds,
\end{eqnarray*}
which verifies (\ref{ap_eq_1}). Therefore, it suffices for us to show that $ \exists \{a_m\}_{m=m_0}^{\infty}$  such that (\ref{ap_eq_6}) and (\ref{ap_eq_7}) holds.

First note that, since $x_m\rightarrow 0$ as $m \rightarrow \infty$, there exists $m_0$ such that $x_m<\frac{\varepsilon}{2g}$ for $m \geq m_0$. Now, by (\ref{nebir})

\begin{eqnarray*}
\left\vert a_{m,R}\right\vert& \leq & \frac{1}{2\pi} \left\vert  \int_{c-\mathbbm{i}R}^{c+\mathbbm{i}R} x_m^s \varepsilon^{d-s} \left(   \underset{j=1}{\overset{d-1 } {\sum }} \frac{h^{s-j}}{s-j} \kappa_j(G)  \right)ds  \right\vert  \\
& +& \frac{1}{2\pi} \left\vert  \int_{c-\mathbbm{i}R}^{c+\mathbbm{i}R} x_m^s \varepsilon^{d-s} \frac{g^{s-d}}{s-d} \kappa_d(G)  ds  \right\vert\\
& +& \frac{1}{2\pi} \left\vert  \int_{c-\mathbbm{i}R}^{c+\mathbbm{i}R} x_m^s \varepsilon^{d-s} \Lambda(s)  ds  \right\vert \\
&=:& I_1(R,m)+I_2(R,m)+I_3(R,m).
\end{eqnarray*}

To estimate $I_1(R,m)$ and $I_2(R,m)$, note the following result (see, \cite[Sec. 3.3, p. 54-55]{edwards}): For $a>0$
\[
\left\vert  \dfrac{1}{2\pi\mathbbm{i}} \int_{a-\mathbbm{i}R}^{a+\mathbbm{i}R} \frac{x^s}{s} ds\right\vert \leq
\left\{
\begin{tabular}{lll}
$\frac{x^a}{\pi R \left\vert \log x \right \vert} $& $\text{ for }$ & $0<x<1 $\\
$K$ & $\text{ for }$ & $x=1$ \\
$\frac{x^a}{\pi R \log x }+res\left( \frac{x^s}{s};0 \right)$ & $\text{ for }$ & $x>1$.%
\end{tabular}
\right.
\]

Henceforth, we will denote by $K$ a finite positive constant whose value might be different at each occurrence. A similar approach shows that for $a>j$ and $R>1$

\[
\left\vert  \dfrac{1}{2\pi\mathbbm{i}} \int_{a-\mathbbm{i}R}^{a+\mathbbm{i}R} \frac{x^s}{s-j} ds\right\vert \leq
\left\{
\begin{tabular}{lll}
$\frac{x^a}{\pi \left\vert \log x \right \vert} $& $\text{ for }$ & $0<x<1 $\\
$K$ & $\text{ for }$ & $x=1$ \\
$\frac{x^a}{\pi \log x }+x^j$ & $\text{ for }$ & $x>1$%
\end{tabular}
\right.
\]
and for $a<j$ and $R>1$
\[
\left\vert  \dfrac{1}{2\pi\mathbbm{i}} \int_{a-\mathbbm{i}R}^{a+\mathbbm{i}R} \frac{x^s}{s-j} ds\right\vert \leq
\left\{
\begin{tabular}{lll}
$\frac{x^a}{\pi \left\vert \log x \right \vert} +x^j$& $\text{ for }$ & $0<x<1 $\\
$K$ & $\text{ for }$ & $x=1$ \\
$\frac{x^a}{\pi \log x }$ & $\text{ for }$ & $x>1$.%
\end{tabular}
\right.
\]
Using the above results, we obtain
\[
 I_1(R,m) \leq I_1(m):=K \left\{
\begin{tabular}{lll}
$ \frac{ \left(\frac{x_m h}{\varepsilon}\right)^c  }{\left\vert \log \left(\frac{x_m h}{\varepsilon}\right) \right \vert}$& $\text{ for }$ & $0<x_m<\frac{\varepsilon}{h} $\\
$1$ & $\text{ for }$ & $x_m=\frac{\varepsilon}{h}$ \\
$\frac{ \left(\frac{x_m h}{\varepsilon}\right)^c}  { \log\left(\frac{x_m h}{\varepsilon}\right) }+ \underset{j=0}{\overset{d-1 } {\sum }} \left(\frac{x_m h}{\varepsilon}\right)^j $ & $\text{ for }$ & $x_m>\frac{\varepsilon}{h}$,%
\end{tabular}
\right.
\]

\[
 I_2(R,m) \leq I_2(m):=K \left\{
\begin{tabular}{lll}
$\frac{ \left(\frac{x_m g}{\varepsilon}\right)^c  }{ \left\vert \log\left(\frac{x_m g}{\varepsilon}\right) \right \vert} +  \left(\frac{x_m g}{\varepsilon}\right)^d$& $\text{ for }$ & $0<x_m<\frac{\varepsilon}{g} $\\
$1$ & $\text{ for }$ & $x_m=\frac{\varepsilon}{g}$ \\
$\frac{ \left(\frac{x_m g}{\varepsilon}\right)^c}  {  \log \left(\frac{x_m g}{\varepsilon} \right) }$ & $\text{ for }$ & $x_m>\frac{\varepsilon}{g}$.%
\end{tabular}
\right.
\]

Recall that for $m \geq m_0$, $x_m <\frac{\varepsilon}{2g}<\frac{\varepsilon}{2h}$. Hence
\[
\underset{m=m_0}{\overset{\infty } {\sum }} I_1(m)+I_2(m) \leq K \underset{x_m < \frac{\varepsilon}{2g}}{\sum } \left( \frac{x_m^c}{\left\vert\log x_m \right\vert }+ x_m^d \right)<\infty,
\]
by (\ref{ap_eq_2}).

We now deal with $I_3(R,m)$. Since $x_m^s \varepsilon^{d-s} \Lambda(s)$ is an entire function, by Cauchy's theorem
\[
\int_{c-\mathbbm{i}R}^{c+\mathbbm{i}R} x_m^s \varepsilon^{d-s} \Lambda(s)  ds=\int_{C_R} x_m^s \varepsilon^{d-s} \Lambda(s)  ds,
\]
where $C_R:C_R(t)=c+R e^{\mathbbm{i}t}$, $-\frac{\pi}{2} \leq t \leq \frac{\pi}{2}$.

Note that for $s=c+R e^{\mathbbm{i}t}$,
\begin{eqnarray*}
\left\vert \Lambda(s) \right\vert& = &\left\vert  \int_h^g u^{s-d-1} \lambda(u) du  \right \vert \\
&\leq& \int_h^g u^{c+R\cos t-d-1} \left\vert \lambda(u)  \right \vert du\\
&\leq& K g^{R \cos t}.
\end{eqnarray*}
Hence,
\begin{eqnarray*}
I_3(R,m)& = &\left\vert \int_{C_R} x_m^s \varepsilon^{d-s} \Lambda(s)  ds \right\vert\\
& =& \left\vert \int_{-\frac{\pi}{2}}^{\frac{\pi}{2}} x_m^{c+R e^{\mathbbm{i}t}} \varepsilon^{d-c-R e^{\mathbbm{i}t}} \Lambda( c+ R e^{\mathbbm{i}t})  \mathbbm{i} R e^{\mathbbm{i}t} dt\right\vert\\
& \leq&K x_m^c R \int_{-\frac{\pi}{2}}^{\frac{\pi}{2}} \left(  \frac{x_m g}{\varepsilon}   \right)^{R \cos t} dt \\
& \leq&K x_m^c R \int_{0}^{\frac{\pi}{2}} \left(  \frac{x_m g}{\varepsilon}   \right)^{R \cos t} dt\\
& \leq&K x_m^c R \int_{0}^{\frac{\pi}{2}} \left(  \frac{x_m g}{\varepsilon}   \right)^{R \sin t} dt.
\end{eqnarray*}

Note that $\sin t \geq \frac{2}{\pi} t $ for $t \in [0,\frac{\pi}{2}]$. Hence, for $x_m<\frac{\varepsilon}{2g}$ we have
\begin{eqnarray*}
I_3(R,m) &\leq & K x_m^c R \int_0^{\frac{\pi}{2}} \left(  \frac{x_m g}{\varepsilon}   \right)^{\frac{2R}{\pi}t} dt\\
&=&K x_m^c R \frac{\pi}{2R \log \left(\frac{x_m g}{\varepsilon} \right)} \left( \left(\frac{x_m g}{\varepsilon}\right)^R -1 \right)\\
&\leq& K \frac{x_m^c}{ \left\vert \log x_m \right\vert},
\end{eqnarray*}
since  $\left( \frac{x_m g}{\varepsilon} \right)^R \rightarrow 0 $ as $R\rightarrow \infty$. Letting $I_3(m)=\frac{K x_m^c}{\left\vert \log x_m\right\vert}$, we obtain $I_3(R,m) \leq I_3(m)$ for all $m \geq m_0$ and
\[
\underset{m=m_0}{\overset{\infty } {\sum }}I_3(m)\leq K \underset{m=m_0}{\overset{\infty } {\sum }}\frac{x_m^c}{\left\vert \log x_m \right\vert}<\infty,
\]
by (\ref{ap_eq_2}).

We set $a_m = I_1(m)+I_2(m)+I_3(m), \text{ } m\geq m_0$. Then both (\ref{ap_eq_6}) and (\ref{ap_eq_7}) are satisfied and our proof is completed.
\end{appendices}

\end{document}